\documentclass{amsart}

\usepackage{tikz, float}
\usepackage{url}

\newtheorem{theorem}{Theorem}[section]
\newtheorem{lemma}[theorem]{Lemma}

\theoremstyle{definition}

\newtheorem{problem}[theorem]{Problem}

\theoremstyle{remark}

\numberwithin{equation}{section}

\begin{document}

\title[Undecidability of Translational Tiling of the Plane with Four Tiles]{Undecidability of Translational Tiling of the Plane with Four Tiles}

%    Information for first author
\author{Chao Yang}
%    Address of record for the research reported here
\address{School of Mathematics and Statistics, Guangdong University of Foreign Studies, Guangzhou, 510006, China}
%    Current address
%\curraddr{Department of Mathematics and Statistics,
%Case Western Reserve University, Cleveland, Ohio 43403}
\email{sokoban2007@163.com, yangchao@gdufs.edu.cn}
%    \thanks will become a 1st page footnote.
%\thanks{The first author was supported by the Research Fund of Guangdong University of Foreign Studies (Nos. 297-ZW200011 and 297-ZW230018), and the National Natural Science Foundation of China (No. 61976104).}

%    Information for second author
\author{Zhujun Zhang}
\address{Big Data Center of Fengxian District, Shanghai, 201499, China}
\email{zhangzhujun1988@163.com}
%\thanks{Support information for the second author.}

%    General info

%\date{August 7, 2024 and, in revised form, December 11, 2024.}

%\dedicatory{This paper is dedicated to our advisors.}


\begin{abstract}
The translational tiling problem, dated back to Wang's domino problem in the 1960s, is one of the most representative undecidable problems in the field of discrete geometry and combinatorics. Ollinger initiated the study of the undecidability of translational tiling with a fixed number of tiles in 2009, and proved that translational tiling of the plane with a set of $11$ polyominoes is undecidable. The number of polyominoes needed to obtain undecidability was reduced from $11$ to $7$ by Yang and Zhang, and then to $5$ by Kim. We show that translational tiling of the plane with a set of $4$ (disconnected) polyominoes is undecidable in this paper.
\end{abstract}

\maketitle

\section{Introduction} 

In all branches of mathematics, there are some basic and fundamental problems that have been shown to be undecidable. Poonen gives a very concise survey for undecidable problems in various fields of mathematics \cite{p14}. Wang's domino tiling problem is a representative one in discrete geometry and combinatorics. A \textit{Wang tile} is a unit square with each edge assigned a color. Given a finite set of Wang tiles (see Figure \ref{fig_w3} for an example), Wang considered the problem of tiling the entire plane with translated copies of the tiles of the set, under the conditions that the tiles must be edge-to-edge and the color of common edges of any two adjacent Wang tiles must be the same \cite{wang61}. Wang asked for a general algorithm to decide the tilability of an arbitrary set of Wang tiles. This is known as \textit{Wang's domino problem}. Berger proved the undecidability of Wang's domino problem in the 1960s~\cite{b66}.

\begin{theorem}[\cite{b66}]\label{thm_berger}
    Wang's domino problem is undecidable.
\end{theorem}

%%% A set of Wang tiles

\begin{figure}[ht]
\begin{center}
\begin{tikzpicture}[scale=0.5]

\draw [ fill=green!50] (0,0)--(2,2)--(2,0)--(0,0);
\draw [ fill=red!50] (0,0)--(2,-2)--(2,0)--(0,0);
\draw [ fill=red!50] (4,0)--(2,2)--(2,0)--(4,0);
\draw [ fill=yellow!50] (4,0)--(2,-2)--(2,0)--(4,0);

\draw [ fill=yellow!50] (6+0,0)--(6+2,2)--(6+2,0)--(6+0,0);
\draw [ fill=blue!50] (6+0,0)--(6+2,-2)--(6+2,0)--(6+0,0);
\draw [ fill=blue!50] (6+4,0)--(6+2,2)--(6+2,0)--(6+4,0);
\draw [ fill=red!50] (6+4,0)--(6+2,-2)--(6+2,0)--(6+4,0);

\draw [ fill=red!50] (12+0,0)--(12+2,2)--(12+2,0)--(12+0,0);
\draw [ fill=yellow!50] (12+0,0)--(12+2,-2)--(12+2,0)--(12+0,0);
\draw [ fill=yellow!50] (12+4,0)--(12+2,2)--(12+2,0)--(12+4,0);
\draw [ fill=green!50] (12+4,0)--(12+2,-2)--(12+2,0)--(12+4,0);

\end{tikzpicture}
\end{center}
\caption{A set of $3$ Wang tiles} \label{fig_w3}
\end{figure}
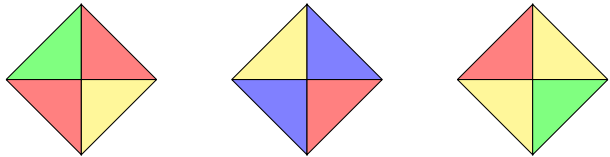

Note the size (i.e., number of tiles) of the set of Wang tiles can be arbitrarily large in Wang's domino problem. If the size of the set of Wang tiles is fixed, then Wang's domino problem is decidable, as there are only a finite number of instances. Ollinger generalizes Wang's domino problem by considering the translational tiling problem with a set of polyominoes of fixed size.

\begin{problem}[Translational Tiling Problem with $k$ Polyominoes] \label{pro_main}
Let $k$ be a fixed positive integer. Is there a general algorithm to decide whether the entire plane can be tiled by translated copies of tiles in an arbitrary set $S$ of $k$ polyominoes?
\end{problem}

If a polyomino can tile the plane, then it can always tile the plane periodically. As a result, the translational tiling problem with a single polyomino (i.e., $k$=1) is decidable \cite{bn91, w15}, even if the tile is disconnected \cite{b20, gt21}. On the other hand, Ollinger gave the first undecidability result with a fixed number of tiles by showing that the Problem \ref{pro_main} is undecidable for $k=11$ \cite{o09}. Because aperiodicity is a necessary condition for undecidability, and the size of the smallest aperiodic set of polyominoes for the translational tiling of the plane is 8 of Ammann's set \cite{ags92} from the 1990s, Ollinger suggested investigating the undecidability of the translational tiling problem with $k$ polyominoes for $k=8,9,10$. Yang and Zhang confirmed that translational tiling problems with $k$ polyominoes are indeed undecidable for $k=8,9,10$ in a series of works \cite{yang23, yang23b, yz24}. It turns out that $k=8$ is not optimal for the undecidability for Problem \ref{pro_main}. The translational tiling problem with $k$ polyominoes was later shown to be undecidable for $k=7$ by Yang and Zhang \cite{yz24d} and $k=5$ by Kim \cite{k25}.

The main contribution of this paper is the undecidability of Problem \ref{pro_main} for $k=4$,.

\begin{theorem}\label{thm_main}
    Translational tiling of the plane with a set of $4$ tiles is undecidable.
\end{theorem}

Following Ollinger's reduction framework, Theorem \ref{thm_main} will be proved by reduction from Wang's domino problem. The colors of Wang tiles are encoded in a binary system consisting of two different shapes of dents or bumps added to a polyomino. To simulate the color matching rules of Wang tiles, all previous works \cite{k25, o09, yang23, yang23b, yz24, yz24d} naturally need two or more tiles to link the two types of dents or bumps, respectively. One of the most important innovations in this paper is to use only one tile to link the different shapes. As a result, the total number of tiles is reduced by one.

Note that some of the tiles that we will construct to prove our main result (Theorem \ref{thm_main}) are not connected. So in this paper, a polyomino can be disconnected. The rest of the paper is organized as follows. Section \ref{sec_bb} introduces the building blocks. Section \ref{sec_4t} completes the definition of four tiles. Section \ref{sec_proof} proves the main result. Section \ref{sec_remark} concludes with a few remarks.

\section{Building Blocks}\label{sec_bb}

A \textit{polyomino} is a finite set of unit squares that are aligned to the same lattice $\mathbb{Z}^2$. So, a polyomino may not be connected in this paper. A $1\times 1$ unit square is called a level-$1$ square. A $7\times 7$ square polyomino is called a level-$2$ square. All building blocks are based on rectangular polyominoes of size $84\times 14$ (of level-$1$ unit squares), which can also be viewed as consisting of $14\times 2$ level-$2$ squares (see Figure \ref{fig_bbb}). Dents and bumps can be added to the north and south sides of the  $84\times 14$ rectangular polyomino to obtain more versatile building blocks. Dents and bumps can be added to both sides (Figures \ref{fig_bb}, \ref{fig_bb_blank}, \ref{fig_bb_2}), or can be added to only one side (Figure \ref{fig_bb_half}).

%%% the only building block

\begin{figure}[ht]
\begin{center}
\begin{tikzpicture}[scale=0.14]

\draw  [fill=gray!20] (0,0)--(0,14)--(84,14)--(84,0)--(0,0);

\foreach \x in {1,...,11}
{
\draw [dashed] (\x*7,0)--(\x*7,14);
}
\draw [dashed] (0,7)--(84,7);

\foreach \x in {1,...,7}
{
\draw (\x,7)--(\x,0);
\draw (7,\x)--(0,\x);
}

\end{tikzpicture}
\end{center}
\caption{A rectangular polyomino of size $14\times 84$.} \label{fig_bbb}
\end{figure}
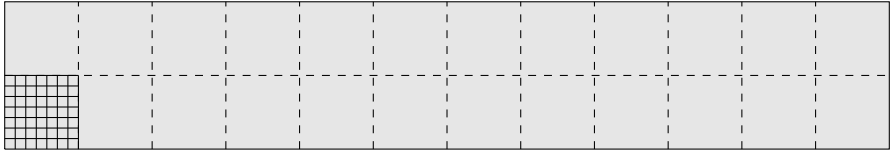

The basic shape of the dents or bumps is the plus-shape polyomino illustrated in Figure \ref{fig_tiny_filler}. This polyomino is also one of the four tiles in the proof of Theorem \ref{thm_main}, and it is called the \textit{tiny filler}. When the dents are not filled by the bumps in a tiling, the tiny filler will assume the responsibility of tiling all the remaining holes of the same plus-shape. We adopt the same shape for the tiny filler as in \cite{k25}. However, the tiny filler can be changed to any other shape as long as it coincides with itself under rotation of $\pi$ and cannot tile the plane with translated copies of itself.

%%% the tiny filler

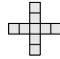
\begin{figure}[ht]
\begin{center}
\begin{tikzpicture}[scale=0.14]

\draw [fill=gray!20] (0,0)--(0,2)--(2,2)--(2,3)--(0,3)--(0,5)--(-1,5)--(-1,3)--(-3,3)--(-3,2)--(-1,2)--(-1,0)--(0,0);

\foreach \x in {-2,...,1}
{
\draw (\x,2)--(\x,3);
}
\foreach \y in {1,...,4}
{
\draw (-1,\y)--(0,\y);
}

\end{tikzpicture}
\end{center}
\caption{The tiny filler (plus-shape).} \label{fig_tiny_filler}
\end{figure}

As we have mentioned, the building blocks can be viewed as two rows of level-$2$ squares, with $12$ level-$2$ squares on each row. We divide the building block into three segments: the \textit{left segment} consists of $5\times 2$ level-$2$ squares on the left, the \textit{right segment} consists of $5\times 2$ level-$2$ squares on the right, and the \textit{middle segment} consists of $2\times 2$ level-$2$ squares in the middle. The level-$2$ squares on the left or right segments are labeled by $A$, $C$, $L$, $M$ or $R$, as illustrated in Figure \ref{fig_bb}.

%%% the only building block

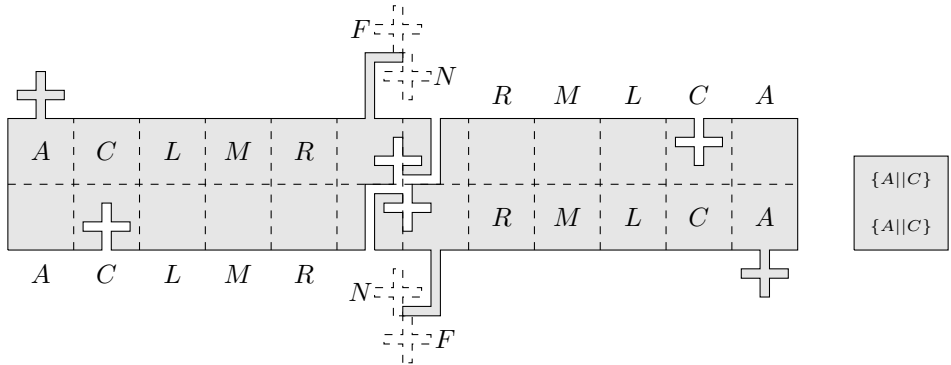
\begin{figure}[ht]
\begin{center}
\begin{tikzpicture}[scale=0.125]

\foreach \x in {7}
\foreach \y in {-84}
{\draw [dashed] (\x+34,105+\y)--(\x+34,107+\y)--(\x+32,107+\y)--(\x+32,108+\y)--(\x+34,108+\y)--(\x+34,110+\y)--(\x+35,110+\y)--(\x+35,108+\y)--(\x+37,108+\y)--(\x+37,107+\y)--(\x+35,107+\y)--(\x+35,105+\y)--(\x+34,105+\y);}
\foreach \x in {8}
\foreach \y in {-89}
{\draw [dashed] (\x+34,105+\y)--(\x+34,107+\y)--(\x+32,107+\y)--(\x+32,108+\y)--(\x+34,108+\y)--(\x+34,110+\y)--(\x+35,110+\y)--(\x+35,108+\y)--(\x+37,108+\y)--(\x+37,107+\y)--(\x+35,107+\y)--(\x+35,105+\y)--(\x+34,105+\y);}
\foreach \x in {7}
\foreach \y in {-112}
{\draw [dashed] (\x+34,105+\y)--(\x+34,107+\y)--(\x+32,107+\y)--(\x+32,108+\y)--(\x+34,108+\y)--(\x+34,110+\y)--(\x+35,110+\y)--(\x+35,108+\y)--(\x+37,108+\y)--(\x+37,107+\y)--(\x+35,107+\y)--(\x+35,105+\y)--(\x+34,105+\y);}
\foreach \x in {8}
\foreach \y in {-117}
{\draw [dashed] (\x+34,105+\y)--(\x+34,107+\y)--(\x+32,107+\y)--(\x+32,108+\y)--(\x+34,108+\y)--(\x+34,110+\y)--(\x+35,110+\y)--(\x+35,108+\y)--(\x+37,108+\y)--(\x+37,107+\y)--(\x+35,107+\y)--(\x+35,105+\y)--(\x+34,105+\y);}

\draw  [fill=gray!20] (0,0)--(0,14)--(3,14)--(3,16)--(1,16)--(1,17)--(3,17)--(3,19)--(4,19)--(4,17)--(6,17)--(6,16)--(4,16)--(4,14)--(38,14)--(38,21)--(42,21)--(42,20)--(39,20)--(39,14)--(45,14)--(45,8)--(42,8)--(42,9)--(44,9)--(44,10)--(42,10)--(42,12)--(41,12)--(41,10)--(39,10)--(39,9)--(41,9)--(41,7)--(38,7)--(38,0)--(11,0)--(11,2)--(13,2)--(13,3)--(11,3)--(11,5)--(10,5)--(10,3)--(8,3)--(8,2)--(10,2)--(10,0)--(0,0);

\draw [fill=gray!20] (39,0)--(39,6)--(42,6)--(42,5)--(40,5)--(40,4)--(42,4)--(42,2)--(43,2)--(43,4)--(45,4)--(45,5)--(43,5)--(43,7)--(46,7)--(46,14)--(73,14)--(73,12)--(71,12)--(71,11)--(73,11)--(73,9)--(74,9)--(74,11)--(76,11)--(76,12)--(74,12)--(74,14)--(84,14)--(84,0)--(81,0)--(81,-2)--(83,-2)--(83,-3)--(81,-3)--(81,-5)--(80,-5)--(80,-3)--(78,-3)--(78,-2)--(80,-2)--(80,0)--(46,0)--(46,-7)--(42,-7)--(42,-6)--(45,-6)--(45,0)--(39,0);

\node at (3.5,10.5) {$A$};
\node at (10.5,10.5) {$C$};
\node at (17.5,10.5) {$L$};
\node at (24.5,10.5) {$M$};
\node at (31.5,10.5) {$R$};

\node at (80.5,16.5) {$A$};
\node at (73.5,16.5) {$C$};
\node at (66.5,16.5) {$L$};
\node at (59.5,16.5) {$M$};
\node at (52.5,16.5) {$R$};

\node at (3.5,-2.5) {$A$};
\node at (10.5,-2.5) {$C$};
\node at (17.5,-2.5) {$L$};
\node at (24.5,-2.5) {$M$};
\node at (31.5,-2.5) {$R$};

\node at (80.5,3.5) {$A$};
\node at (73.5,3.5) {$C$};
\node at (66.5,3.5) {$L$};
\node at (59.5,3.5) {$M$};
\node at (52.5,3.5) {$R$};

\node at (37.5,23.5) {$F$};
\node at (46.5,18.5) {$N$};

\node at (37.5,-4.5) {$N$};
\node at (46.5,-9.5) {$F$};

\foreach \x in {1,...,11}
{
\draw [dashed] (\x*7,0)--(\x*7,14);
}
\draw [dashed] (0,7)--(84,7);

\draw [fill=gray!20] (90,0)--(100,0)--(100,10)--(90,10)--(90,0);

\node at (95,7.5) {\tiny $\{A||C\}$};
\node at (95,2.5) {\tiny $\{A||C\}$};

\end{tikzpicture}
\end{center}
\caption{The linker: a building block with label $\{A||C\}$ on the north and $\{A||C\}$ on the south (left), and its symbolic representation (right).} \label{fig_bb}
\end{figure}

For each level-$2$ square (of left or right segments) in the first row (resp. second row), a dent or a bump can be added to the north (resp. south). Two types of dents or bumps can be added to the building blocks. The first type of dents or bumps can be added directly to a level-$2$ square on the left or right segments. The second type of dents or bumps is added to the middle segment of the building blocks. An $L$-shape bump is added to the sixth level-$2$ square in the first row, and the seventh level-$2$ square in the second row. An $L$-shape dent is added to the seventh level-$2$ square in the first row, and the sixth level-$2$ square in the second row. A $L$-shape dent or bump is called a \textit{handle}. The locations for adding the dented or bumped handles in the middle segment are fixed, and these handles are always required. Because a bumped handle can only be matched by a dented handle, they enforce the alignment of building blocks in a tiling. In addition, two plus-shape dents are also always attached to the dented handle as illustrated in Figure \ref{fig_bb}. For the bump part, at most one plus-shape bump can be attached to the bumped handle at the locations depicted in dashed lines in Figure \ref{fig_bb}. The two locations that a plus-shape bump can be attached to the bumped handle are denoted by $N$ and $F$ (see Figure \ref{fig_bb}), respectively. Compare to the dents or bumps of the first type, the plus-shape dents of the second type are not added to the level-$2$ squares directly, but indirectly through the handles. Therefore, the two types of dents or bumps form two independent matching systems which are used for two distinct purposes. As we shall see later, the first type enforces the overall rigid tiling pattern, and the second type enforces the simulation of the matching rules of Wang tiles.

Below is a summary of all locations that a plus-shape dent or bump can be added. Different locations have different intended purposes, which will be revealed in later sections.

\begin{itemize}
    \item $A$: for linking two \textbf{adjacent} building blocks that are encoding colors;
    \item $C$: for indicating the building block is encoding \textbf{color} of Wang tiles (but the actual encoding job is done by $F$ and $N$, see below);
    \item $M$: for \textbf{marking} the two ends of a sequence of building blocks that encodes a colored edge of a Wang tile;
    \item $L$: for selecting a simulated a Wang tile out of the encoder (the encoder is one of the $4$ tiles that will be introduced in the next section) from the \textbf{left};
    \item $R$: for selecting a simulated a Wang tile out of the encoder from the \textbf{right};
    \item $F$: attached to the \textbf{far} side of the handle;
    \item $N$: attached to the \textbf{near} side of the handle; $F$ and $N$ are used for encoding colors in binary system.
\end{itemize}

Now we introduce a notation for labeling all the dents or bumps on one side (either north side or south side) of the building blocks, which generalizes the notation used in \cite{yz25}. A north side (resp. south) is denoted by a set with three separate parts $\{X_1,\cdots,X_r|Y_1,\cdots,Y_s|Z_1,\cdots,Z_t\}$, where $X_i$ means there is a plus-shape bump at location $X_i$ of the left segment (resp. right segment), $Y_i$ means there is a plus-shape bump at location $Y_i$ on the middles segment, and $Z_i$ means there is a plus-shape dent at location $Z_i$ of the right segment (resp. left segment). Give a label $S=\{X_1,\cdots,X_r|Y_1,\cdots,Y_s|Z_1,\cdots,Z_t\}$, let $S_L=\{X_1,\cdots,X_r\}$, $S_M=\{Y_1,\cdots,Y_s\}$, and $S_R=\{Z_1,\cdots,Z_t\}$. Note that the notation records the dents and bumps in different directions for the north side and south side: from left to right on the north side and from right on the south side.

With this notation, the building block illustrated in Figure \ref{fig_bb} is labeled $T=\{A||C\}$ (the second part is empty, namely $T_M=\emptyset$) on both the north and south sides, and is symbolically represented by a gray square with two labels on the north and south sides (on the right of Figure \ref{fig_bb}). This building block is called \textit{linker}, which is also one of the four tiles that will be introduced in the next section. One of the most important innovative techniques introduced in this paper is the use of a single linker, compared to the two linkers in all previous proofs for undecidability results following Ollinger's framework.

%%% the only building block  /// HALF version

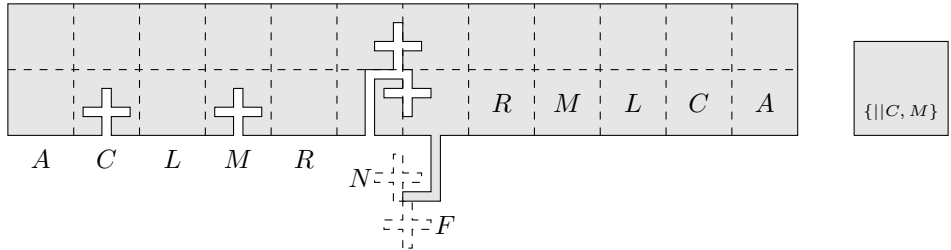
\begin{figure}[ht]
\begin{center}
\begin{tikzpicture}[scale=0.125]

\foreach \x in {7}
\foreach \y in {-112}
{\draw [dashed] (\x+34,105+\y)--(\x+34,107+\y)--(\x+32,107+\y)--(\x+32,108+\y)--(\x+34,108+\y)--(\x+34,110+\y)--(\x+35,110+\y)--(\x+35,108+\y)--(\x+37,108+\y)--(\x+37,107+\y)--(\x+35,107+\y)--(\x+35,105+\y)--(\x+34,105+\y);}
\foreach \x in {8}
\foreach \y in {-117}
{\draw [dashed] (\x+34,105+\y)--(\x+34,107+\y)--(\x+32,107+\y)--(\x+32,108+\y)--(\x+34,108+\y)--(\x+34,110+\y)--(\x+35,110+\y)--(\x+35,108+\y)--(\x+37,108+\y)--(\x+37,107+\y)--(\x+35,107+\y)--(\x+35,105+\y)--(\x+34,105+\y);}

%\draw  [fill=gray!20] (0,0)--(0,14)--(3,14)--(3,16)--(1,16)--(1,17)--(3,17)--(3,19)--(4,19)--(4,17)--(6,17)--(6,16)--(4,16)--(4,14)--(38,14)--(38,21)--(42,21)--(42,20)--(39,20)--(39,14)--(45,14)--(45,8)--(42,8)--(42,9)--(44,9)--(44,10)--(42,10)--(42,12)--(41,12)--(41,10)--(39,10)--(39,9)--(41,9)--(41,7)--(38,7)--(38,0)--(11,0)--(11,2)--(13,2)--(13,3)--(11,3)--(11,5)--(10,5)--(10,3)--(8,3)--(8,2)--(10,2)--(10,0)--(0,0);

\draw [fill=gray!20] (39,0)--(39,6)--(42,6)--(42,5)--(40,5)--(40,4)--(42,4)--(42,2)--(43,2)--(43,4)--(45,4)--(45,5)--(43,5)--(43,7)--(42,7)--(42,9)--(44,9)--(44,10)--(42,10)--(42,12)--(41,12)--(41,10)--(39,10)--(39,9)--(41,9)--(41,7)--(38,7)--(38,0)--(25,0)--(25,2)--(27,2)--(27,3)--(25,3)--(25,5)--(24,5)--(24,3)--(22,3)--(22,2)--(24,2)--(24,0)--(11,0)--(11,2)--(13,2)--(13,3)--(11,3)--(11,5)--(10,5)--(10,3)--(8,3)--(8,2)--(10,2)--(10,0)--(0,0)--(0,14)--(84,14)--(84,0)--(81,0)--(46,0)--(46,-7)--(42,-7)--(42,-6)--(45,-6)--(45,0)--(39,0);

\node at (3.5,-2.5) {$A$};
\node at (10.5,-2.5) {$C$};
\node at (17.5,-2.5) {$L$};
\node at (24.5,-2.5) {$M$};
\node at (31.5,-2.5) {$R$};

\node at (80.5,3.5) {$A$};
\node at (73.5,3.5) {$C$};
\node at (66.5,3.5) {$L$};
\node at (59.5,3.5) {$M$};
\node at (52.5,3.5) {$R$};

\node at (37.5,-4.5) {$N$};
\node at (46.5,-9.5) {$F$};

\foreach \x in {1,...,11}
{
\draw [dashed] (\x*7,0)--(\x*7,14);
}
\draw [dashed] (0,7)--(84,7);

\draw [fill=gray!20] (90,0)--(100,0)--(100,10)--(90,10)--(90,0);

\node at (95,7.5) {};
\node at (95,2.5) {\tiny $\{||C,M\}$};

\end{tikzpicture}
\end{center}
\caption{A building block with label $\{||C,M\}$ on the south (left), and its symbolic representation (right).} \label{fig_bb_half}
\end{figure}

%%% the blank building block ********************************

\begin{figure}[ht]
\begin{center}
\begin{tikzpicture}[scale=0.125]

\foreach \x in {7}
\foreach \y in {-84}
{\draw [dashed] (\x+34,105+\y)--(\x+34,107+\y)--(\x+32,107+\y)--(\x+32,108+\y)--(\x+34,108+\y)--(\x+34,110+\y)--(\x+35,110+\y)--(\x+35,108+\y)--(\x+37,108+\y)--(\x+37,107+\y)--(\x+35,107+\y)--(\x+35,105+\y)--(\x+34,105+\y);}
\foreach \x in {8}
\foreach \y in {-89}
{\draw [dashed] (\x+34,105+\y)--(\x+34,107+\y)--(\x+32,107+\y)--(\x+32,108+\y)--(\x+34,108+\y)--(\x+34,110+\y)--(\x+35,110+\y)--(\x+35,108+\y)--(\x+37,108+\y)--(\x+37,107+\y)--(\x+35,107+\y)--(\x+35,105+\y)--(\x+34,105+\y);}
\foreach \x in {7}
\foreach \y in {-112}
{\draw [dashed] (\x+34,105+\y)--(\x+34,107+\y)--(\x+32,107+\y)--(\x+32,108+\y)--(\x+34,108+\y)--(\x+34,110+\y)--(\x+35,110+\y)--(\x+35,108+\y)--(\x+37,108+\y)--(\x+37,107+\y)--(\x+35,107+\y)--(\x+35,105+\y)--(\x+34,105+\y);}
\foreach \x in {8}
\foreach \y in {-117}
{\draw [dashed] (\x+34,105+\y)--(\x+34,107+\y)--(\x+32,107+\y)--(\x+32,108+\y)--(\x+34,108+\y)--(\x+34,110+\y)--(\x+35,110+\y)--(\x+35,108+\y)--(\x+37,108+\y)--(\x+37,107+\y)--(\x+35,107+\y)--(\x+35,105+\y)--(\x+34,105+\y);}

\draw  [fill=gray!20] (0,0)--(0,14)--(38,14)--(38,21)--(42,21)--(42,20)--(39,20)--(39,14)--(45,14)--(45,8)--(42,8)--(42,9)--(44,9)--(44,10)--(42,10)--(42,12)--(41,12)--(41,10)--(39,10)--(39,9)--(41,9)--(41,7)--(38,7)--(38,0)--(0,0);

\draw [fill=gray!20] (39,0)--(39,6)--(42,6)--(42,5)--(40,5)--(40,4)--(42,4)--(42,2)--(43,2)--(43,4)--(45,4)--(45,5)--(43,5)--(43,7)--(46,7)--(46,14)--(84,14)--(84,0)--(81,0)--(46,0)--(46,-7)--(42,-7)--(42,-6)--(45,-6)--(45,0)--(39,0);

\node at (3.5,10.5) {$A$};
\node at (10.5,10.5) {$C$};
\node at (17.5,10.5) {$L$};
\node at (24.5,10.5) {$M$};
\node at (31.5,10.5) {$R$};

\node at (80.5,16.5) {$A$};
\node at (73.5,16.5) {$C$};
\node at (66.5,16.5) {$L$};
\node at (59.5,16.5) {$M$};
\node at (52.5,16.5) {$R$};

\node at (3.5,-2.5) {$A$};
\node at (10.5,-2.5) {$C$};
\node at (17.5,-2.5) {$L$};
\node at (24.5,-2.5) {$M$};
\node at (31.5,-2.5) {$R$};

\node at (80.5,3.5) {$A$};
\node at (73.5,3.5) {$C$};
\node at (66.5,3.5) {$L$};
\node at (59.5,3.5) {$M$};
\node at (52.5,3.5) {$R$};

\node at (37.5,23.5) {$F$};
\node at (46.5,18.5) {$N$};

\node at (37.5,-4.5) {$N$};
\node at (46.5,-9.5) {$F$};

\foreach \x in {1,...,11}
{
\draw [dashed] (\x*7,0)--(\x*7,14);
}
\draw [dashed] (0,7)--(84,7);

\draw [fill=gray!20] (90,0)--(100,0)--(100,10)--(90,10)--(90,0);

\node at (95,7.5) {\tiny $\{||\}$};
\node at (95,2.5) {\tiny $\{||\}$};

\end{tikzpicture}
\end{center}
\caption{A blank building block with label $\{||\}$ on the north and $\{||\}$ on the south (left), and its symbolic representation (right).} \label{fig_bb_blank}
\end{figure}
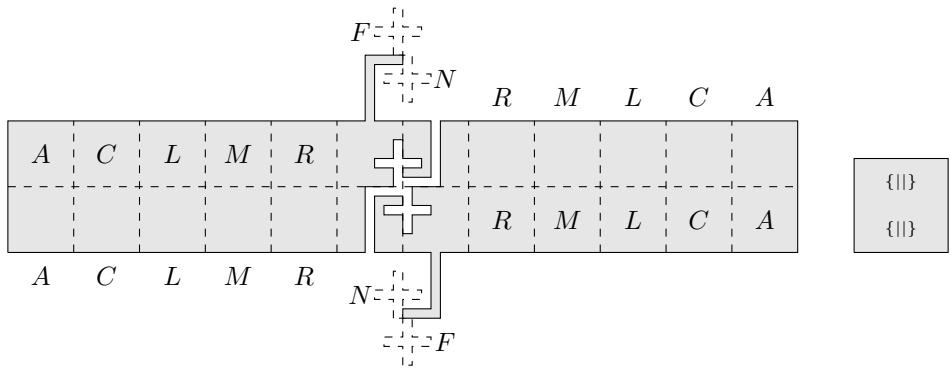

%%% the only building block  (for encoding colors)

\begin{figure}[ht]
\begin{center}
\begin{tikzpicture}[scale=0.125]

%\foreach \x in {7}
%\foreach \y in {-84}
%{\draw [dashed] (\x+34,105+\y)--(\x+34,107+\y)--(\x+32,107+\y)--(\x+32,108+\y)--(\x+34,108+\y)--(\x+34,110+\y)--(\x+35,110+\y)--(\x+35,108+\y)--(\x+37,108+\y)--(\x+37,107+\y)--(\x+35,107+\y)--(\x+35,105+\y)--(\x+34,105+\y);}
\foreach \x in {8}
\foreach \y in {-89}
{\draw [dashed] (\x+34,105+\y)--(\x+34,107+\y)--(\x+32,107+\y)--(\x+32,108+\y)--(\x+34,108+\y)--(\x+34,110+\y)--(\x+35,110+\y)--(\x+35,108+\y)--(\x+37,108+\y)--(\x+37,107+\y)--(\x+35,107+\y)--(\x+35,105+\y)--(\x+34,105+\y);}
%\foreach \x in {7}
%\foreach \y in {-112}
%{\draw [dashed] (\x+34,105+\y)--(\x+34,107+\y)--(\x+32,107+\y)--(\x+32,108+\y)--(\x+34,108+\y)--(\x+34,110+\y)--(\x+35,110+\y)--(\x+35,108+\y)--(\x+37,108+\y)--(\x+37,107+\y)--(\x+35,107+\y)--(\x+35,105+\y)--(\x+34,105+\y);}
\foreach \x in {8}
\foreach \y in {-117}
{\draw [dashed] (\x+34,105+\y)--(\x+34,107+\y)--(\x+32,107+\y)--(\x+32,108+\y)--(\x+34,108+\y)--(\x+34,110+\y)--(\x+35,110+\y)--(\x+35,108+\y)--(\x+37,108+\y)--(\x+37,107+\y)--(\x+35,107+\y)--(\x+35,105+\y)--(\x+34,105+\y);}

\draw  [fill=gray!20] (0,0)--(0,14)--(10,14)--(10,16)--(8,16)--(8,17)--(10,17)--(10,19)--(11,19)--(11,17)--(13,17)--(13,16)--(11,16)--(11,14)--(38,14)--(38,21)--(41,21)--(41,23)--(39,23)--(39,24)--(41,24)--(41,26)--(42,26)--(42,24)--(44,24)--(44,23)--(42,23)--(42,20)--(39,20)--(39,14)--(45,14)--(45,8)--(42,8)--(42,9)--(44,9)--(44,10)--(42,10)--(42,12)--(41,12)--(41,10)--(39,10)--(39,9)--(41,9)--(41,7)--(38,7)--(38,0)--(4,0)--(4,2)--(6,2)--(6,3)--(4,3)--(4,5)--(3,5)--(3,3)--(1,3)--(1,2)--(3,2)--(3,0)--(0,0);

\draw [fill=gray!20] (39,0)--(39,6)--(42,6)--(42,5)--(40,5)--(40,4)--(42,4)--(42,2)--(43,2)--(43,4)--(45,4)--(45,5)--(43,5)--(43,7)--(46,7)--(46,14)--(80,14)--(80,12)--(78,12)--(78,11)--(80,11)--(80,9)--(81,9)--(81,11)--(83,11)--(83,12)--(81,12)--(81,14)--(84,14)--(84,0)--(74,0)--(74,-2)--(76,-2)--(76,-3)--(74,-3)--(74,-5)--(73,-5)--(73,-3)--(71,-3)--(71,-2)--(73,-2)--(73,0)--(46,0)--(46,-7)--(41,-7)--(41,-5)--(39,-5)--(39,-4)--(41,-4)--(41,-2)--(42,-2)--(42,-4)--(44,-4)--(44,-5)--(42,-5)--(42,-6)--(45,-6)--(45,0)--(39,0);

\node at (3.5,10.5) {$A$};
\node at (10.5,10.5) {$C$};
\node at (17.5,10.5) {$L$};
\node at (24.5,10.5) {$M$};
\node at (31.5,10.5) {$R$};

\node at (80.5,16.5) {$A$};
\node at (73.5,16.5) {$C$};
\node at (66.5,16.5) {$L$};
\node at (59.5,16.5) {$M$};
\node at (52.5,16.5) {$R$};

\node at (3.5,-2.5) {$A$};
\node at (10.5,-2.5) {$C$};
\node at (17.5,-2.5) {$L$};
\node at (24.5,-2.5) {$M$};
\node at (31.5,-2.5) {$R$};

\node at (80.5,3.5) {$A$};
\node at (73.5,3.5) {$C$};
\node at (66.5,3.5) {$L$};
\node at (59.5,3.5) {$M$};
\node at (52.5,3.5) {$R$};

\node at (37.5,23.5) {$F$};
\node at (46.5,18.5) {$N$};

\node at (37.5,-4.5) {$N$};
\node at (46.5,-9.5) {$F$};

\foreach \x in {1,...,11}
{
\draw [dashed] (\x*7,0)--(\x*7,14);
}
\draw [dashed] (0,7)--(84,7);

\draw [fill=gray!20] (90,0)--(100,0)--(100,10)--(90,10)--(90,0);

\node at (95,7.5) {\tiny $\{C|F|A\}$};
\node at (95,2.5) {\tiny $\{C|N|A\}$};

\end{tikzpicture}
\end{center}
\caption{A building block with label $\{C|F|A\}$ on the north and $\{C|N|A\}$ on the south (left), and its symbolic representation (right).} \label{fig_bb_2}
\end{figure}
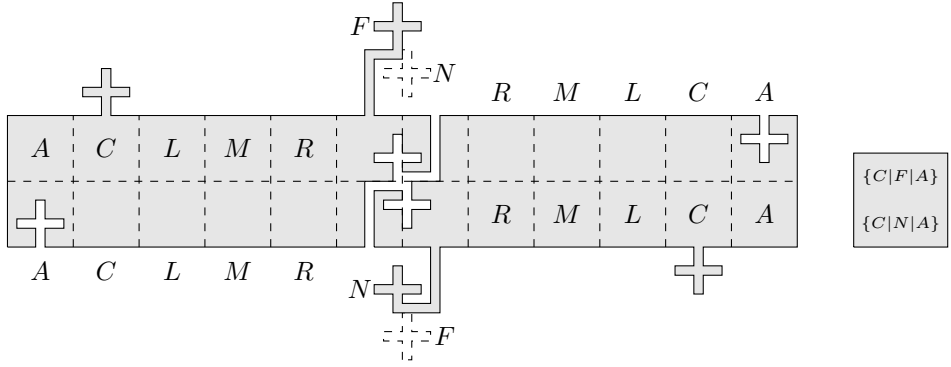

When the building blocks are put together to form bigger tiles, sometimes only one of the two sides (the north and south sides) is exposed as the boundary of the bigger tile. In this case, the dents or bumps are added only to one side (the boundary side) of the building block. Figure \ref{fig_bb_half} illustrates a building block with dents and bumps only on the south side, and its symbolic representation. Two more examples of building blocks are illustrated in Figure \ref{fig_bb_blank} and Figure \ref{fig_bb_2}. Figure \ref{fig_bb_blank} illustrates a building block that is labeled with the empty notation $\{||\}$ on both sides, that is, no extra dents or bumps are added to either side.  By comparing the north sides of the two building blocks illustrated in Figure \ref{fig_bb_half} and Figure \ref{fig_bb_blank}, we can notice the slight difference between a side without a label and a side with the empty label $\{||\}$. Figure \ref{fig_bb_2} illustrates a type of building block that will be used to define the encoder, one of the four tiles, in the next section. Note that the location (either $F$ or $N$) of bump added to the middle segment encodes one bit of information.

With the notations for labeling the building blocks, we can state the following lemma on the conditions under which two building blocks can be placed next to each other.

\begin{lemma}\label{lem_bb}
Let $B_1$ be a building block with a label $S$ on the south side, and $B_2$ be a building block with a label $T$ on the north side. Then $B_1$ can be placed adjacent to the north of $B_2$ without overlap if and only if $T_L\subseteq S_R$ and $S_L\subseteq T_R$.
\end{lemma}

\begin{proof}
The $L$-shape bump (i.e., the handle) can only be matched by the $L$-shape dent, so the two building blocks $B_1$ and $B_2$ must be aligned when placed next to each other in north-south direction. In other words, the left (resp. middle, right) segment  of $B_1$ is adjacent to the left (resp. middle, right) segment of $B_2$. To avoid overlap, where there is a bump in one building block, there must be a dent in the other building block. The middle segments will not overlap with just two adjacent building blocks $B_1$ and $B_2$. So it suffices to make sure that the left and right segments do not overlap. It is straightforward to check that this is equivalent to $T_L\subseteq S_R$ and $S_L\subseteq T_R$.
\end{proof}

The conditions in Lemma \ref{lem_bb} only guarantee that there are no overlaps, but there may be some plus-shape holes left. As mentioned before, all these holes can be filled with the \textit{tiny filler} (Figure \ref{fig_tiny_filler}). 

In the remainder of this paper, symbolic representations and labels are used to depict the building blocks. Note that all the symbolic representation of the building blocks are not to scale.

\section{The Set of Four Tiles}\label{sec_4t}

Given a set of Wang tiles, we will construct a set of four tiles: a tiny filler, a linker, a locator, and an encoder. The tiny filler and linker have already been defined in the previous section. The locator and encoder are introduced in this section. The tiny filler and linker are fixed in size. In other words, their sizes do not depend on the number of tiles or edge colors in a set of Wang tiles. However, the size of the locator and encoder increases along with the size of the set of Wang tiles.

%%% the locator

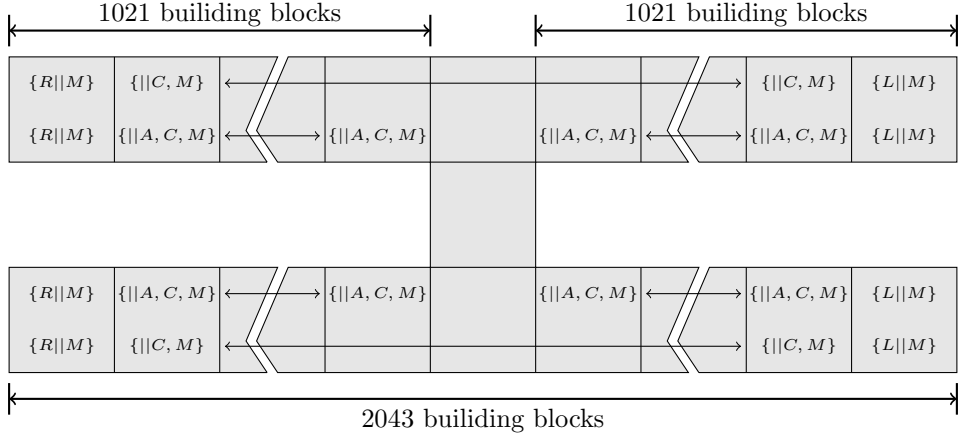
\begin{figure}[ht]
\begin{center}
\begin{tikzpicture}[scale=0.7]

\draw [ fill=gray!20] (0,0)--(0,2)--(2.9,2)--(2.5,2.6)--(3.1,4)--(-4.7,4)--(-5.3,2.6)--(-4.9,2)--(-2,2)--(-2,0)--(-4.7,0)--(-5.3,-1.4)--(-4.9,-2)--(2.9,-2)--(2.5,-1.4)--(3.1,0)--(0,0);

\draw [ fill=gray!20] (3.1,-2)--(2.7,-1.4)--(3.3,0)--(8,0)--(8,-2)--(3.1,-2);
\draw [ fill=gray!20] (3.1,2)--(2.7,2.6)--(3.3,4)--(8,4)--(8,2)--(3.1,2);
\draw [fill=gray!20] (-5.1,2)--(-5.5,2.6)--(-4.9,4)--(-10,4)--(-10,2)--(-5.1,2);
\draw [fill=gray!20] (-5.1,-2)--(-5.5,-1.4)--(-4.9,0)--(-10,0)--(-10,-2)--(-5.1,-2);

\draw (0,0)--(-2,0);
\draw (0,2)--(-2,2);

\foreach \x in {-8,-6,-4,-2,0,2,4,6}
\foreach \y in {-2,2}
{
\draw (\x,\y)--(\x,\y+2);
}

%%% labels (selectors  L )
\foreach \x in {-9}
\foreach \y in {-2,-1,2,3}
{
\node at (\x,\y+0.5) {\tiny ${\{R||M\}}$}; 
}

%%% labels (selectors  R )
\foreach \x in {7}
\foreach \y in {-2,-1,2,3}
{
\node at (\x,\y+0.5) {\tiny ${\{L||M\}}$}; 
}

%%% labels others
\foreach \x in {-7,-3,1,5}
\foreach \y in {-1,2}
{
\node at (\x,\y+0.5) {\tiny ${\{||A,C,M\}}$}; 
}
\foreach \x in {-7,1}
\foreach \y in {-1,2}
{
\draw [<->] (\x+1.1,\y+0.5)--(\x+2.9,\y+0.5);
}

\foreach \x in {-7,5}
\foreach \y in {-2,3}
{
\node at (\x,\y+0.5) {\tiny ${\{||C,M\}}$}; 
}
\foreach \x in {-7}
\foreach \y in {-2,3}
{
\draw [<->] (\x+1.1,\y+0.5)--(\x+10.9,\y+0.5);
}

\draw [thick] (-10,-2.8)--(-10,-2.2);
\draw [thick] (8,-2.8)--(8,-2.2);
\draw [<->,thick] (-10,-2.5)--(8,-2.5);

\node at (-1,-2.9) {$2043$ builiding blocks};

\draw [thick] (-10,4.2)--(-10,4.8);
\draw [thick] (-2,4.2)--(-2,4.8);
\draw [thick] (0,4.2)--(0,4.8);
\draw [thick] (8,4.2)--(8,4.8);

\draw [<->,thick] (-10,4.5)--(-2,4.5);
\draw [<->,thick] (0,4.5)--(8,4.5);
\node at (4,4.8) {$1021$ builiding blocks};
\node at (-6,4.8) {$1021$ builiding blocks};

\end{tikzpicture}
\end{center}
\caption{The locator} \label{fig_locator}
\end{figure}

The \textit{locator} is illustrated in Figure \ref{fig_locator}, and consists of two rows of $2043$ building blocks and another building block that connects the two rows right in the middle. For each north or south side (of a building block) that appears on the outer boundary of the locator, dents and bumps are added to that side, and their labels are shown in Figure \ref{fig_locator}. The two left (resp. right) most building blocks of the locator, whose north and south sides are labeled $\{R||M\}$ (resp. $\{L||M\}$), are called the \textit{right selectors} (resp. \textit{left selectors}). Together, the four selectors will choose one of the simulated Wang tiles from each encoder. All other north or south sides are labeled either $\{||C,M\}$ if they appear on the top or bottom boundary of the locator, or $\{||A,C,M\}$ if they appear in the concave part of the locator. These non-selector building blocks do not serve any special purposes, and their dents are just to make sure that everything else of the tiling works fine without overlaps.

The locator illustrated in Figure \ref{fig_locator} corresponds to the set of Wang tiles in Figure \ref{fig_w3}. In general, for a set of $n$ Wang tiles with $m$ colors and $t=\lceil \log_2 m\rceil $ (in the remainder of this paper, the parameters $n$, $m$ and $t$ are used to describe the size of a general set of Wang tiles), each of the two rows of the locator will consist of $2^{3n}(t+2)-2t-1$ building blocks. The structure of the general locator is the same as the example in Figure \ref{fig_locator}, by adding more non-selector building blocks to lengthen the two rows, and the building block that connects the two rows remains at the center.

%%% the encoding parts of a Wang tile

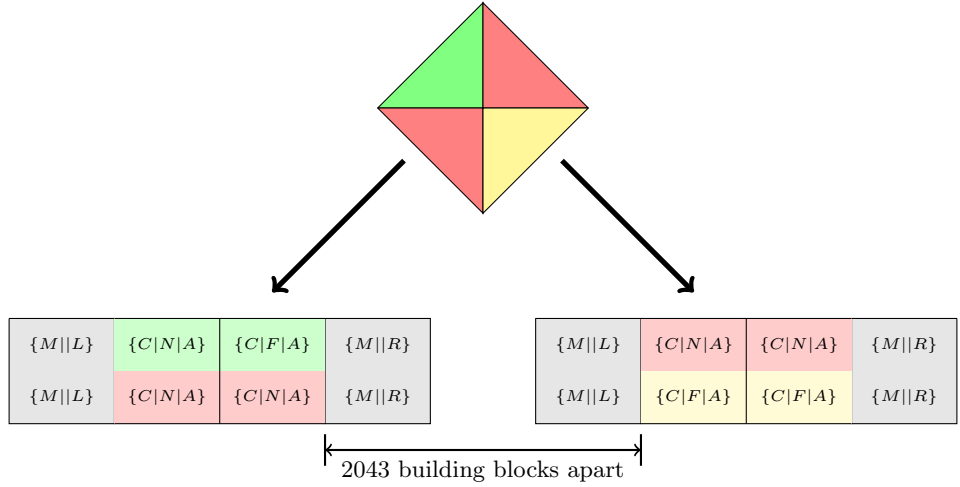
\begin{figure}[ht]
\begin{center}
\begin{tikzpicture}[scale=0.7]

\draw [ fill=green!50] (0,0)--(2,2)--(2,0)--(0,0);
\draw [ fill=red!50] (0,0)--(2,-2)--(2,0)--(0,0);
\draw [ fill=red!50] (4,0)--(2,2)--(2,0)--(4,0);
\draw [ fill=yellow!50] (4,0)--(2,-2)--(2,0)--(4,0);

\draw [ fill=gray!20] (-7,-4)--(-5,-4)--(-5,-6)--(-7,-6)--(-7,-4);
\draw [ fill=gray!20] (-1,-4)--(1,-4)--(1,-6)--(-1,-6)--(-1,-4);
\draw [ fill=green!20,draw=none] (-5,-4)--(-1,-4)--(-1,-5)--(-5,-5)--(-5,-4);
\draw [ fill=red!20,draw=none] (-5,-6)--(-1,-6)--(-1,-5)--(-5,-5)--(-5,-6);
\draw (-1,-4)--(-5,-4);
\draw (-1,-6)--(-5,-6);
\draw (-3,-6)--(-3,-4);

\draw [ fill=gray!20] (3,-4)--(5,-4)--(5,-6)--(3,-6)--(3,-4);
\draw [ fill=gray!20] (11,-4)--(9,-4)--(9,-6)--(11,-6)--(11,-4);
\draw [ fill=red!20,draw=none] (5,-4)--(9,-4)--(9,-5)--(5,-5)--(5,-4);
\draw [ fill=yellow!20,draw=none] (5,-6)--(9,-6)--(9,-5)--(5,-5)--(5,-6);
\draw (9,-4)--(5,-4);
\draw (9,-6)--(5,-6);
\draw (7,-6)--(7,-4);

%%% labels
\foreach \x in {-6,4}
\foreach \y in {-5,-6}
{
\node at (\x,\y+0.5) {\tiny ${\{M||L\}}$}; 
}

%%% labels
\foreach \x in {0,10}
\foreach \y in {-5,-6}
{
\node at (\x,\y+0.5) {\tiny ${\{M||R\}}$}; 
}

% RED
\foreach \x in {6,8}
\foreach \y in {-5}
{
\node at (\x,\y+0.5) {\tiny ${\{C|N|A\}}$}; 
}
\foreach \x in {-2,-4}
\foreach \y in {-6}
{
\node at (\x,\y+0.5) {\tiny ${\{C|N|A\}}$}; 
}

% YELLOW
\foreach \x in {6,8}
\foreach \y in {-6}
{
\node at (\x,\y+0.5) {\tiny ${\{C|F|A\}}$}; 
}

%Green
\foreach \x in {-4}
\foreach \y in {-5}
{
\node at (\x,\y+0.5) {\tiny ${\{C|N|A\}}$}; 
}
\foreach \x in {-2}
\foreach \y in {-5}
{
\node at (\x,\y+0.5) {\tiny ${\{C|F|A\}}$}; 
}

\draw [->,line width=2pt] (3.5,-1)--(6,-3.5);
\draw [->,line width=2pt] (0.5,-1)--(-2,-3.5);

\draw [thick] (-1,-6.2)--(-1,-6.8);
\draw [thick] (5,-6.2)--(5,-6.8);
\draw [<->,thick] (-1,-6.5)--(5,-6.5);
\node at (2,-6.9) {\small $2043$ building blocks apart};

\end{tikzpicture}
\end{center}
\caption{The two portions encoding a Wang tile} \label{fig_encoder_portion}
\end{figure}

The encoder consists of only one row of consecutive building blocks. So, in a sense, the overall shape of the encoder is simpler than that of the locator. However, the layout of different building blocks in the encoder is more complex. Before introducing the complete construction of the encoder, we first define the partial structures that are intended to simulate Wang tiles. A Wang tile from the set illustrated in Figure \ref{fig_w3} is simulated by two portions (the left portion and the right portion) of consecutive building blocks as shown in Figure \ref{fig_encoder_portion}. Each portion consists of two \textit{encoding building blocks} and two \textit{marker} building blocks. The north or south side of an encoding building block is labeled either $\{C|N|A\}$ or $\{C|F|A\}$. The four colors, red, green, blue and yellow in the example of Figure \ref{fig_w3} are encoded by two consecutive labels $\{C|N|A\}\{C|N|A\}$, $\{C|N|A\}\{C|F|A\}$, $\{C|F|A\}\{C|N|A\}$ and $\{C|F|A\}\{C|F|A\}$, respectively. The northwest and southwest edges of a Wang tile are encoded in the left portion, and the northeast and southeast sides are encoded in the right portion. Together, the two portions encode a complete Wang tile. The distance between the encoding building blocks of the two portions is $2043$ building blocks, which is equal to the length of the locator. The distance is crucial in forming the rigid tiling pattern, as we will see in the next section. The left (resp. right) marker is labeled $\{M||L\}$ (resp. $\{M||R\}$) on both the north and south sides. 

In general, a portion consists of $t+2$ building blocks: $t$ building blocks for encoding color edges of Wang tiles, and $2$ markers. The encoding building blocks of the two portions that simulate the same Wang tile are $2^{3n}(t+2)-2t-1$ building blocks apart.

%%%%%%%%%%%%%%%%%
%%% the encoder
%%%%%%%%%%%%%%%%%

\begin{figure}[ht]
\begin{center}
\begin{tikzpicture}[scale=0.5]

\draw [fill=gray!20] (0,0)--(24,0)--(24,1)--(0,1)--(0,0);
\draw (8,0)--(8,1);
\draw (16,0)--(16,1);

\draw (1,0)--(1,1);
\draw (2,0)--(2,1);
\draw (7,0)--(7,1);
\draw (17,0)--(17,1);
\draw (18,0)--(18,1);
\draw (23,0)--(23,1);

\draw [thick] (8,1.2)--(8,1.8);
\draw [thick] (16,1.2)--(16,1.8);

\draw [thick] (0,-.2)--(0,-.8);
\draw [thick] (8,-.2)--(8,-.8);
\draw [thick] (16,-.2)--(16,-.8);
\draw [thick] (24,-.2)--(24,-.8);

\draw [<->,thick] (0,-.5)--(8,-.5);
\draw [<->,thick] (8,1.5)--(16,1.5);
\draw [<->,thick] (16,-.5)--(24,-.5);
\node at (4,-1) {$1024$ building blocks};
\node at (12,2) {$1021$ building blocks};
\node at (20,-1) {$1024$ building blocks};

\node at (0.5,0.5) {$1$};
\node at (1.5,0.5) {$2$};
\node at (4.5,0.5) {$\cdots \cdots$};
\node at (7.5,0.5) {\tiny $256$};
\node at (16.5,0.5) {$1$};
\node at (17.5,0.5) {$2$};
\node at (20.5,0.5) {$\cdots \cdots$};
\node at (23.5,0.5) {\tiny $256$};

\end{tikzpicture}
\end{center}
\caption{The encoder} \label{fig_encoder}
\end{figure}
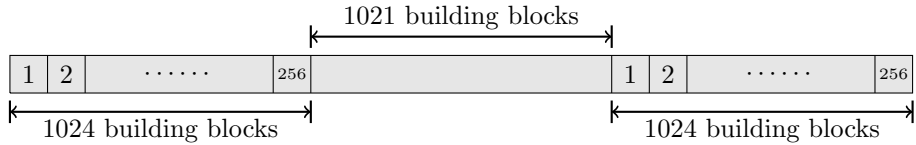

With the portions for encoding the Wang tiles, it is now ready to define the encoder. The \textit{encoder} that encodes the set of Wang tiles in Figure \ref{fig_w3} consists of $3$ segments: a \textit{left encoding segment} of 1024 building blocks, a \textit{central padding segment} of 1021 building blocks, and a \textit{right encoding segment} of 1024 building blocks. The left or right segments are divided into $256$ slots, and each slot consists of $4$ building blocks, which is exactly the length of a portion for simulating half of a Wang tile. The slots are indexed by positive integers from $1$ to $256$ from left to right. Only slots indexed with a power of $2$ are used to encode Wang tiles. Each Wang tile is encoded on the encoder three times in duplicate. In more detail, the left (resp. right) portion for encoding the first Wang tile resides at slots $1$, $2$, $4$ of the left (resp. right) segment, the left (resp. right) portion for encoding the second Wang tile resides in slots $8$, $16$, $32$ of the left (resp. right) segment, and the left (resp. right) portion for encoding the third Wang tile resides in slots $64$, $128$, $256$ of the left (resp. right) segment. Because each Wang tile has three simulated copies in the encoder, only the two portions with the same slot indices (e.g., slot 1 of the left segment and slot 1 of the right segment) are regarded as the belonging to the same simulated Wang tile. All other building blocks except the encoding portions are called \textit{padding building blocks}. In other words, the padding building blocks include the central padding segment and all the building blocks that reside in slots not indexed by the powers of $2$. Padding building blocks are labeled $\{||\}$ (see Figure \ref{fig_bb_blank}) on both the north and south sides.

In general, for a set of Wang tiles with parameters $n$, $m$ and $t$, the left or right encoding segment consists of $2^{3n-1}(t+2)$ building blocks, and the central padding segment consists of $2^{3n-1}(t+2)-t-1$ building blocks. Recall that in the general situation, a portion has $t+2$ building blocks. Therefore, a slot in general also consists of $t+2$ building blocks, and the left or right encoding segment is divided into $2^{3n-1}$ slots indexed by integers from $1=2^0$ to $2^{3n-1}$. The left portion (resp. right portion) of the $i$-th Wang tile appears on slots $2^{3(i-1)}$, $2^{3(i-1)+1}$ and $2^{3(i-1)+2}$ of the left encoding segment (resp. right encoding segment), for $1\leq i \leq n$. All other building blocks (i.e., building blocks of slots that are not indexed by powers of $2$, and building blocks on the padding segment) are just the blank building block. It can be checked that the distance between the encoding building blocks of the two portions that simulate the same Wang tile is $2^{3n}(t+2)-2t-1$ building blocks.

The sizes of the encoder and locator are mutually affected by each other. To conclude this section, we summarize the order in which the sizes are determined by the parameters ($n$, $m$, $t$) of the set of Wang tiles. First, the left or right encoding segment of the encoder is determined by $n$ and $t$. A portion (equivalently, a slot) must have $t+2$ building blocks to encode the colors. Because simulated Wang tiles reside only in slots indexed by powers of $2$ and each simulated Wang tile appears three times, the largest index of the slots should be $2^{3n-1}$. So the length of the one (left or right) encoding segment is 
$$2^{3n-1}(t+2)$$
building blocks. Second, the length of one concave part of the encoder should not be long enough to contain an encoding segment of the encoder completely, and must leave the encoding building block of at least one slot outside. So the length on one concave part of the locator should be $2^{3n-1}(t+2)-t-1$. As a consequence, the length of the locator is determined to be 
$$2  \big( 2^{3n-1}(t+2)-t-1 \big) +1 = 2^{3n}(t+2)-2t-1 $$
building blocks. Third, we compute the length of the padding segment of the encoder such that the distance between the encoding building blocks of two portion (two slots) of the same simulated Wang tiles is exactly the length of the locator. So the padding segment is
$$\big (  2^{3n}(t+2)-2t-1 \big ) - \big (  2^{3n-1}(t+2)-t \big ) = 2^{3n-1}(t+2)-t-1 $$
building blocks. Finally, by adding up the lengths of the two encoding segment and one padding segment, the total length of the encoder is 
$$2\big ( 2^{3n-1}(t+2) \big) + 2^{3n-1}(t+2)-t-1 = 3\big ( 2^{3n-1}(t+2) \big) -t -1 $$
building blocks.

\section{Rigid Tiling Pattern with Local Flexibility}\label{sec_proof}

We prove the main result in this section.

\begin{proof}[Proof of Theorem \ref{thm_main}]
Given an arbitrary set of Wang tiles, we have already constructed a set of $4$ tiles, a tiny filler, a linker, a locator, and an encoder, in the previous section. To complete the proof, we will show that the set of $4$ tiles can tile the plane if and only if the corresponding set of Wang tiles can tile the plane.

\begin{itemize}

\item \textbf{The tiny filler alone cannot tile the plane.} By the pseudo-hexagon characterization of translational single tile that can tile the plane \cite{bn91}, it is obvious that the tiny filler cannot tile the plane with translated copies of itself.

\item \textbf{The locators must be used.} Since tiny filler cannot tile the plane by itself, so the other three kinds of tiles must be used. If the encoder is used, then the locator must be used. Because the markers (labeled $\{M||L\}$ or $\{M||R\}$) in the encoder can only be matched by the selectors (labeled $\{L||M\}$ or $\{R||M\}$) in the locator by Lemma \ref{lem_bb}. If the linker (labeled $\{A||C\}$) is used, then either the encoder or the locator must be used in order to match the bump at the location $A$ of the linker. So in any case, the locator must be used.

\item \textbf{The locators form a rigid lattice pattern.} Put a locator any where in the plane (see the locator marked $1$ in Figure \ref{fig_pattern}). Because the left selector in the locator can only be matched by the left marker in the encoder, so an encoder (marked $2$ in Figure \ref{fig_pattern}) must be placed somewhere inside the concave area of locator $1$. By the same reason, the first right marker (of encoder $2$) that exposed outside the locator $1$  must also be matched by other locators, so the locators $3$ and $4$ are place above and below encoder $2$, respectively. Apply the above arguments repeatedly to newly added locators, we obtained an infinite rigid lattice of locators as shown in Figure \ref{fig_pattern}. Note that the part of encoder $2$ illustrated in dark orange in Figure \ref{fig_pattern} (the part exposed outside locators $1$, $3$, and $4$) is exactly the $t$ encoding building blocks (excluding the markers) of the left encoding portion of a simulated Wang tile. As mentioned in previous section, the distance between the encoding building blocks of the two portions of the same simulated Wang tile is equal to the length of the locator, so the other part (of encoder $2$) that is exposed outside the locators (illustrated in purple in Figure \ref{fig_pattern}) must be the $t$ encoding building blocks of the right portion of the same simulated Wang tile as the dark orange part. Therefore, exactly one simulated Wang tiles of every encoder is exposed outside the rigid lattice of locators.

%%% the rigidity 

\begin{figure}[ht]
\begin{center}
\begin{tikzpicture}[scale=0.2]

\foreach \x in {0,32}
\foreach \y in {0}
{
\draw [ fill=gray!20] (\x+0,\y+0)--(\x+15,\y+0)--(\x+15,\y+1)--(\x+8,\y+1)--(\x+8,\y+2)--(\x+15,\y+2)--(\x+15,\y+3)--(\x+0,\y+3)--(\x+0,\y+2)--(\x+7,\y+2)--(\x+7,\y+1)--(\x+0,\y+1)--(\x+0,\y+0);
}

\foreach \x in {-16,16,48}
\foreach \y in {-2,2}
{
\draw [ fill=gray!20] (\x+0,\y+0)--(\x+15,\y+0)--(\x+15,\y+1)--(\x+8,\y+1)--(\x+8,\y+2)--(\x+15,\y+2)--(\x+15,\y+3)--(\x+0,\y+3)--(\x+0,\y+2)--(\x+7,\y+2)--(\x+7,\y+1)--(\x+0,\y+1)--(\x+0,\y+0);
}

\draw [fill=orange!20] (9,1)--(34,1)--(34,2)--(9,2)--(9,1);

\foreach \x in {19}
\foreach \y in {2}
{
\draw [fill=orange!20] (\x+9,1+\y)--(\x+34,1+\y)--(\x+34,2+\y)--(\x+9,2+\y)--(\x+9,1+\y);
}

\foreach \x in {20}
\foreach \y in {-2}
{
\draw [fill=orange!20] (\x+9,1+\y)--(\x+34,1+\y)--(\x+34,2+\y)--(\x+9,2+\y)--(\x+9,1+\y);
}

\foreach \x in {-17}
\foreach \y in {2}
{
\draw [fill=orange!20] (\x+9,1+\y)--(\x+34,1+\y)--(\x+34,2+\y)--(\x+9,2+\y)--(\x+9,1+\y);
}

\foreach \x in {-12}
\foreach \y in {-2}
{
\draw [fill=orange!20] (\x+9,1+\y)--(\x+34,1+\y)--(\x+34,2+\y)--(\x+9,2+\y)--(\x+9,1+\y);
}

\node at (7.5,1) {\tiny $1$};
\node at (18,1.5) {\tiny $2$};
\node at (23.5,3) {\tiny $3$};
\node at (23.5,0) {\tiny $4$};

\draw [fill=orange!50,draw=none] (15,2)--(16,2)--(16,1)--(15,1)--(15,2);

\draw [fill=violet!50,draw=none] (32,2)--(31,2)--(31,1)--(32,1)--(32,2);

\end{tikzpicture}
\end{center}
\caption{The rigid tiling pattern (not to scale)} \label{fig_pattern}
\end{figure}
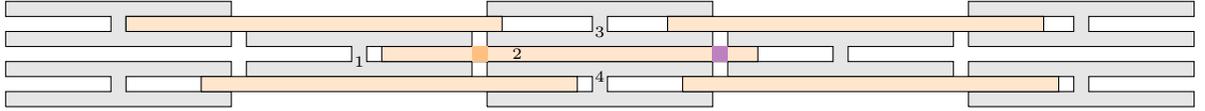

\item \textbf{The encoders are flexible locally.} In contrast to the rigidity of the lattice form by the locators, each encoder has a local flexibility to shift left or right within the horizontal zone between the locators. The markers in the encoders, the selectors in the locators and the length of the locators make sure that the two parts of every encoder that exposed outside the locators are the encoding building blocks of the same simulated Wang tile. This flexibility simulates the choice of a Wang tile in each place.

\item \textbf{The locators and the encoders can be compatible.} Still, we have to check the rigid lattice of locators and flexible encoders within in the lattice illustrated in Figure \ref{fig_pattern} do not overlap with each other. By examining every pair of adjacent building blocks between the locators and encoders, we notice that there is only one case that might cause a conflict, as illustrated in Figure \ref{fig_adjacent_f} where the dark orange part represent the encoding building blocks of the encoders. In Figure \ref{fig_adjacent_f}, slot $i$ of the right encoding segment of the top encoder is vertically aligned with slot $i$ of the left encoding segment of the bottom encoder, for each $i$. In particular, the encoding building blocks (in dark orange) are all aligned vertically. In the locations marked by red dots in Figure \ref{fig_adjacent_f}, $F$ bumps (resp. $N$ bumps) of the encoding building blocks from above and the $N$ bumps (resp. $F$ bumps) from below may overlap.

To avoid the possible overlap in Figure \ref{fig_adjacent_f}, we make use of the local flexibility of the encoders, and shift one of them if needed to keep away from the situation that the right encoding segment of the top encoder is totally aligned with the left encoding segment of the bottom encoder. The resulted local configuration is illustrated in Figure \ref{fig_adjacent}. Let the exposed encoding building blocks of the top and bottom encoders belong to slot $i$ and $j$, respectively. Because the two encoding segments are not aligned, so $i\neq j$. We claim that the two exposed encoding building blocks are the only building blocks (among the right encoding segment of the top encoder and the left segment of the bottom encoder) that are vertically aligned. Suppose to the contrary that there is another pair of encoding slots, slot $i'$ from the top and slot $j'$ from the bottom, are vertically aligned. Then $i-i'=j-j'$. However, all the encoding slots are powers of $2$. Under this condition, $i-i'=j-j'$ implies $i=j$ and $i'=j'$, which is a contradiction.

%%% the flexibility (two encoders) false /// possible conflict 

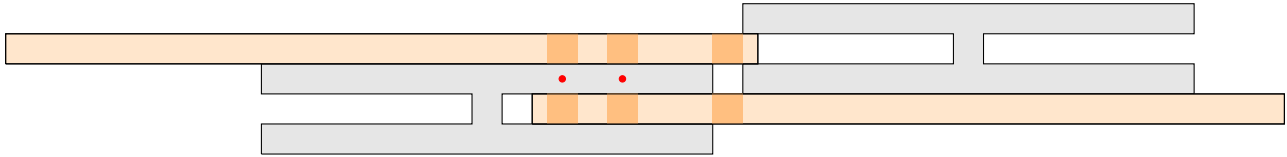
\begin{figure}[ht]
\begin{center}
\begin{tikzpicture}[scale=0.4]

\foreach \x in {0}
\foreach \y in {0}
{
\draw [ fill=gray!20] (\x+0,\y+0)--(\x+15,\y+0)--(\x+15,\y+1)--(\x+8,\y+1)--(\x+8,\y+2)--(\x+15,\y+2)--(\x+15,\y+3)--(\x+0,\y+3)--(\x+0,\y+2)--(\x+7,\y+2)--(\x+7,\y+1)--(\x+0,\y+1)--(\x+0,\y+0);
}

\foreach \x in {16}
\foreach \y in {2}
{
\draw [ fill=gray!20] (\x+0,\y+0)--(\x+15,\y+0)--(\x+15,\y+1)--(\x+8,\y+1)--(\x+8,\y+2)--(\x+15,\y+2)--(\x+15,\y+3)--(\x+0,\y+3)--(\x+0,\y+2)--(\x+7,\y+2)--(\x+7,\y+1)--(\x+0,\y+1)--(\x+0,\y+0);
}

\draw [fill=orange!20] (9,1)--(34,1)--(34,2)--(9,2)--(9,1);

\draw [fill=orange!50,draw=none] (15,1)--(16,1)--(16,2)--(15,2)--(15,1);
\draw [fill=orange!50,draw=none] (11.5,1)--(12.5,1)--(12.5,2)--(11.5,2)--(11.5,1);
\draw [fill=orange!50,draw=none] (10.5,1)--(9.5,1)--(9.5,2)--(10.5,2)--(10.5,1);
\draw  (9,1)--(34,1)--(34,2)--(9,2)--(9,1);

\foreach \x in {-17.5}
\foreach \y in {2}
{
\draw [fill=orange!20] (\x+9,1+\y)--(\x+34,1+\y)--(\x+34,2+\y)--(\x+9,2+\y)--(\x+9,1+\y);
}
\draw [fill=orange!50,draw=none] (11.5,3)--(12.5,3)--(12.5,4)--(11.5,4)--(11.5,3);
\draw [fill=orange!50,draw=none] (15,3)--(16,3)--(16,4)--(15,4)--(15,3);
\draw [fill=orange!50,draw=none] (10.5,3)--(9.5,3)--(9.5,4)--(10.5,4)--(10.5,3);

\foreach \x in {-17.5}
\foreach \y in {2}
{
\draw (\x+9,1+\y)--(\x+34,1+\y)--(\x+34,2+\y)--(\x+9,2+\y)--(\x+9,1+\y);
}

\filldraw[red] (10,2.5) circle (3pt);

\filldraw[red] (12,2.5) circle (3pt);

\end{tikzpicture}
\end{center}
\caption{Two adjacent encoders possibly overlap (not to scale).} \label{fig_adjacent_f}
\end{figure}

The reason that each simulated Wang tiles appear three times in the encoder is to avoid possible overlap (as illustrated in Figure \ref{fig_adjacent_f}) without change of the exposed simulated Wang tiles. Note that encoders on the same horizontal row will never overlap with each other. So we can fix the location of a row of encoders, and determine the locations of all the rest encoders row by row upwards and downwards. In this order, each newly added encoder is adjacent to exactly two encoders whose locations have been fixed. Without loss of generality, assume the newly added encoder is adjacent to two encoders below it. Therefore, the left encoding segment (resp. right encoding segment) the newly added encoder cannot be totally aligned to the right encoding segment (resp. left encoding segment) of the encoder in the southwest direction (resp. southeast direction). On the other hand, the newly added encoder has three different locations (recall that each simulated Wang tile has three copies in an encoder) to choose from without changing the exposed simulated Wang tiles. By pigeonhole principle, there exists a location that avoids totally aligned of encoding segments with the two encoders below it.

%%% the flexibility (two encoders)

\begin{figure}[ht]
\begin{center}
\begin{tikzpicture}[scale=0.4]

\foreach \x in {0}
\foreach \y in {0}
{
\draw [ fill=gray!20] (\x+0,\y+0)--(\x+15,\y+0)--(\x+15,\y+1)--(\x+8,\y+1)--(\x+8,\y+2)--(\x+15,\y+2)--(\x+15,\y+3)--(\x+0,\y+3)--(\x+0,\y+2)--(\x+7,\y+2)--(\x+7,\y+1)--(\x+0,\y+1)--(\x+0,\y+0);
}

\foreach \x in {16}
\foreach \y in {2}
{
\draw [ fill=gray!20] (\x+0,\y+0)--(\x+15,\y+0)--(\x+15,\y+1)--(\x+8,\y+1)--(\x+8,\y+2)--(\x+15,\y+2)--(\x+15,\y+3)--(\x+0,\y+3)--(\x+0,\y+2)--(\x+7,\y+2)--(\x+7,\y+1)--(\x+0,\y+1)--(\x+0,\y+0);
}

\draw [fill=orange!20] (9,1)--(34,1)--(34,2)--(9,2)--(9,1);

\draw [fill=orange!50,draw=none] (15,1)--(16,1)--(16,2)--(15,2)--(15,1);
\draw [fill=orange!50,draw=none] (11.5,1)--(12.5,1)--(12.5,2)--(11.5,2)--(11.5,1);
\draw [fill=orange!50,draw=none] (10.5,1)--(9.5,1)--(9.5,2)--(10.5,2)--(10.5,1);
\draw  (9,1)--(34,1)--(34,2)--(9,2)--(9,1);

\foreach \x in {-14}
\foreach \y in {2}
{
\draw [fill=orange!20] (\x+9,1+\y)--(\x+34,1+\y)--(\x+34,2+\y)--(\x+9,2+\y)--(\x+9,1+\y);
}
\draw [fill=orange!50,draw=none] (16,3)--(15,3)--(15,4)--(16,4)--(16,3);
\draw [fill=orange!50,draw=none] (19.5,3)--(18.5,3)--(18.5,4)--(19.5,4)--(19.5,3);
\draw [fill=orange!50,draw=none] (14,3)--(13,3)--(13,4)--(14,4)--(14,3);

\foreach \x in {-14}
\foreach \y in {2}
{
\draw (\x+9,1+\y)--(\x+34,1+\y)--(\x+34,2+\y)--(\x+9,2+\y)--(\x+9,1+\y);
}

\node at (15.5,3.5) {$i$};
\node at (15.5,1.5) {$j$};

\end{tikzpicture}
\end{center}
\caption{Two adjacent encoders compatible with each other (not to scale).} \label{fig_adjacent}
\end{figure}
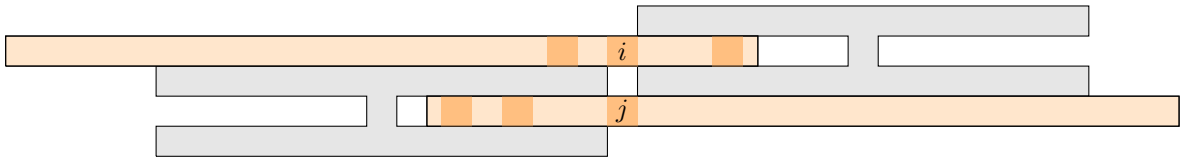

For two adjacent exposed sequence of encoding building blocks, say $i$ and $j$ in Figure \ref{fig_adjacent}, the label sequence of south side of building blocks of $i$ must be the same as the label sequence of the north side of building blocks of $j$, in order to avoid overlap. This is equivalent to the two adjacent simulated edge of Wang tiles are of the same color.

%To make the one-linker construction feasible, another important technique that we have applied is the nonequal-distance placement of the encoding portions within the encoder (see Figure \ref{fig_adjacent}). 

\item \textbf{The gaps are filled by linkers and tiny fillers.} As long as there is no overlaps in the previous step, it is easy to check all the remaining gaps can be filled by the linkers and the tiny fillers.
\end{itemize} 
The above arguments have shown that any tiling of the set of $4$ tiles simulates a tiling of the corresponding set of Wang tiles. Conversely, any tiling of the set of Wang tile can also be converted to a tiling of the set of $4$ tiles. By Theorem \ref{thm_berger}, translational tiling of the plane with $4$ (disconnected) polyominoes is undecidable. \end{proof}

As mentioned in the first section of this paper, our overall strategy to reduce the number of tiles is to merge the two linkers in Ollinger's method into one linker. This strategy is made possible only with the help of two novel techniques. The first technique is the unequal-distance placement of the encoding portions within the encoder. The second technique is to add redundancy by simulating each Wang tile three times in the encoder.

\section{Conclusion}\label{sec_remark}

The translational tiling problem with $k$ polyominoes (Problem \ref{pro_main}) is the $2$-dimensional case of the more general problem of translational tiling of $\mathbb{Z}^n$ with a set of $k$ tiles (both $n$ and $k$ are fixed).

\begin{problem}[Translational tiling of $\mathbb{Z}^n$ with a set of $k$ tiles] \label{pro_gen}
A tile is a finite subset of $\mathbb{Z}^n$. Let $k$ and $n$ be fixed positive integers. Is there a general algorithm to decide whether $\mathbb{Z}^n$ can be tiled by translated copies of tiles in an arbitrary set $S$ of $k$ tiles?
\end{problem}

The two remarkable results \cite{gt24a,gt24b} of Greenfeld and Tao suggest that Problem \ref{pro_gen} may be undecidable for some fixed parameters $(n,k)=(n,1)$. Problem \ref{pro_gen} is known to be undecidable for $(n,k)=(3,4)$ and $(n,k)=(4,3)$ \cite{yz24e}. By combining the techniques in \cite{yz24e} and the techniques (i.e, unequal-distance placement and duplications of the simulated Wang tiles) introduced in this paper, we can also improve the undecidability result from $(n,k)=(4,3)$ to $(n,k)=(3,3)$, which is prepared in another paper.

For dimension $2$, it remains open to determine whether the translational tiling problem is undecidable for $(n,k)=(2,3)$ or $(n,k)=(2,2)$. Also, is it possible to prove the undecidability of $(n,k)=(2,4)$ with $4$ connected tiles?

%\section*{Acknowledgements}
%The first author was supported by the Research Fund of Guangdong University of Foreign Studies (Nos. 297-ZW200011 and 297-ZW230018), and the National Natural Science Foundation of China (No. 61976104).

%\section*{Data Availability Statement}
%This paper has no associated data.

\end{document}